\documentclass[12pt]{amsart}
\usepackage{amssymb,amsmath}
\usepackage[dvips]{graphicx}
\usepackage{pslatex}
\usepackage{amsmath,amscd}
\usepackage{amsthm}
\usepackage{overpic}
\usepackage{color}

\theoremstyle{plain} \newtheorem{thm}{Theorem}[section]
\newtheorem{cor}[thm]{Corollary} \newtheorem{prop}[thm]{Proposition}
\newtheorem{lemma}[thm]{Lemma} 

\newtheorem{question}[thm]{Question} 

\newtheorem*{namedtheorem}{\theoremname}
\newcommand{\theoremname}{testing}

\theoremstyle{definition} \newtheorem{defn}[thm]{Definition}

\theoremstyle{remark} \newtheorem{remark}[thm]{Remark}

\newcommand{\ZZ}{{\mathbb Z}}
\newcommand{\Sth}{\mathbb{S}^3}

\newcommand{\interior}[1]{%
  {\kern0pt#1}^{\mathrm{o}}%
}
 
\begin{document}

\author{Christian Millichap}
\address{Department of Mathematics\\ 
Furman University\\ 
Greenville, SC 29613}
\email{christian.millichap@furman.edu}

\author{Rolland Trapp}
\address{Department of Mathematics\\ 
California State University\\ 
San Bernardino, CA 92407}
\email{rtrapp@csusb.edu}

\title{Symmetry groups of hyperbolic links and their complements}

\begin{abstract}
We explicitly construct a sequence of hyperbolic links $\{ L_{4n} \}$ where the number of symmetries of each $\Sth \setminus L_{4n}$ that  are not induced by symmetries of the pair $(\Sth, L_{4n})$ grows linearly with $n$. Specifically,   $[Sym(\mathbb{S}^{3} \setminus L_{4n}) : Sym(\mathbb{S}^{3}, L_{4n})] =8n \rightarrow \infty$ as $n \rightarrow \infty$.  For this construction, we start with a family of minimally twisted chain links, $\{ C_{4n} \}$, where $Sym(\Sth, C_{4n})$ and $Sym(\Sth \setminus C_{4n})$ coincide and grow linearly with $n$. We then perform a particular type of homeomorphism on $\Sth \setminus C_{4n}$ to produce another link complement $\Sth \setminus L_{4n}$ where we can uniformly bound $|Sym(\Sth, L_{4n})|$ using a combinatorial condition based on linking number. A more general result  highlighting how to control symmetry groups of hyperbolic links is provided, which has potential for further application.
\end{abstract}

\maketitle

\section{Introduction}
\label{sec:intro}

Symmetry groups associated with a link $L$ and its complements $\mathbb{S}^{3} \setminus L$ naturally arise as objects of interest in low-dimensional topology. Here, we will consider two particular symmetry groups: the group of homeomorphisms of the pair $(\Sth, L)$ up to isotopy, denoted $Sym(\Sth, L)$, and  the group of self-homeomorphisms  of $\Sth \setminus L$ up to isotopy, denoted  $Sym(\Sth \setminus L)$.  Furthermore, we will exclusively be interested in the symmetry groups of a hyperbolic link, that is, a link $L \subset \Sth$ whose complement admits a complete metric of constant negative curvature. In this case, Mostow--Prasad rigidity implies that every symmetry of $\Sth \setminus L$ induces a unique isometry on its corresponding hyperbolic structure, and so,  $Sym(\Sth \setminus L) = Isom(\Sth \setminus L)$. Moving forward, we will use $Sym(\Sth \setminus L)$ to represent both of these groups. Analyzing symmetry groups of links via isometries of their complements has proven to be a useful approach; see \cite[Theorem 1]{PaPo2009}, \cite[Theorem 6.1]{MeMiTr2020} for some examples. Since any self-homeomorphism of $(\Sth, L)$ induces a homeomorphism of $\Sth \setminus L$, we have that $Sym(\Sth, L) \subset Sym(\Sth \setminus L)$. Thus, it is natural to ask the following question raised by Paoluzzi--Porti in \cite{PaPo2009}:  What is the relationship between symmetries of a hyperbolic link and isometries of its complement?

Explicit relationships between symmetry groups of hyperbolic links and isometry groups of their respective complements have been established for certain special classes of hyperbolic links.  For any hyperbolic knot $K$, we always have $Sym(\Sth, K) = Sym(\Sth \setminus K)$ since knots are determined by their complements \cite{GL}. In addition, the symmetry group of a hyperbolic knot is always a finite cyclic or dihedral group; this follows from the solution of The Smith Conjecture \cite{Mo1984} and a proof can be found in \cite{Lu1992}. The behavior of symmetry groups of certain hyperbolic links with multiple components can also exhibit a basic structure. For instance, the work of Sakuma--Weeks \cite{SaWe} combined with Gu\'{e}ritaud \cite{Gu:thesis} classified the symmetry groups of $2$-bridge link complements (which include infinitely many links with two components), where only a few groups of order eight or less were possible. These groups coincide with the symmetry groups of their respective links; see Millichap--Worden  \cite[Proposition 4.2]{MiWo2016} for a short proof of the equivalence. In addition, Meyer--Millichap--Trapp \cite{MeMiTr2020} analyzed  symmetry groups of flat fully augmented pretzel links, which can have arbitrary many components. The symmetry groups of these links and their respective complements also coincide.

At the same time, symmetry groups of hyperbolic links with multiple components  can exhibit far more variation. In contrast to hyperbolic knots,  $Sym(\Sth, L)$ can be strictly contained in $Sym(\Sth \setminus L)$, when $L$ admits multiple components. In \cite{HeWe1992}, Henry--Weeks used SnapPea to compute the symmetry groups of hyperbolic links and their respective complements with nine or fewer crossings. Their work revealed a handful of hyperbolic links where these groups did not coincide, though for any such hyperbolic link $L$ with nine or fewer crossings, they found that $[Sym(\mathbb{S}^{3} \setminus L) : Sym(\mathbb{S}^{3}, L)] \leq 4$. While this was partially due to the low complexity of the examples involved, the authors do not know of any other work in the literature that explicitly examines the behavior of $[Sym(\mathbb{S}^{3} \setminus L) : Sym(\mathbb{S}^{3}, L)]$ for hyperbolic links where this index is not one.

The work of Paoluzzi--Porti further highlights this contrast. While symmetry groups of knots can only be cyclic or dihedral, \cite[Theorem 1]{PaPo2009} shows that for any finite group $G$, there exists a hyperbolic link $L \subset \Sth$ such that $G \cong Sym^{+}(\Sth \setminus L)$, with $G$ acting freely;  $Sym^{+}(\Sth \setminus L)$ denotes the subgroup of orientation-preserving symmetries of $Sym(\Sth \setminus L)$. In addition, this result implies that there exist sets of hyperbolic links $\{ K_j \}$ where $[Sym(\Sth \setminus K_j): Sym(\Sth, K_j)] \rightarrow \infty$ as $j \rightarrow \infty$. Here, we provide a brief sketch of this argument that was communicated to the authors in a referee report on a previous version of this paper. Geometrization implies that for a hyperbolic link $L$, $Sym^{+}(\Sth, L)$ is a finite subgroup of $SO(4) \cong \textbf{S}^{3} \times_{\ZZ/2\ZZ} \textbf{S}^{3}$, where $\textbf{S}^{3}$ denotes the multiplicative group of unit quaternions. Furthermore,   $SO(3) \cong \textbf{S}^{3}/G$, where $G$ is an order two subgroup of $\textbf{S}^{3}$, and so, finite subgroups of $SO(4)$ are (up to taking a quotient by at most an order two subgroup) products of two finite subgroups of $SO(3)$. Since finite subgroups of $SO(3)$ are cyclic, dihedral, or one of the three symmetry groups of the Platonic solids, we just need to choose  links $K_{j}$ (via  \cite[Theorem 1]{PaPo2009}) where any finite  subgroups of $SO(4)$ would have arbitrarily large index in $Sym^{+}(\Sth \setminus K_j)$ as $j \rightarrow \infty$. For instance, choosing links such that $Sym^{+}(\Sth \setminus K_j)  \cong (\ZZ/n\ZZ)^{j}$, for $j \geq 3$ and $n$ sufficiently large will result in $[Sym^{+}(\Sth \setminus K_j): Sym^{+}(\Sth, K_j)] \geq n^{j-2}.$ Similar conclusions can be made for the full symmetry groups. 

While the work of Paoluzzi-Porti implies that such sets of hyperbolic links exist, there work does not give an immediate construction.  In response, we provide an explicit construction showing this index can be arbitrarily large using simple topological and combinatorial tools that are independent of their work. Our work also generalizes to a method for constructing other classes of links where their symmetry groups are controlled by basic combinatorial data; see Theorem \ref{thm:symcontrol}. In addition, our work  leads to some interesting questions about symmetry groups of fully augmented links and their complements that could be attacked using some of the tools developed in this paper; see Section \ref{sec:concluding}.  

In the following theorem, $D_{4n}$ denotes the symmetry group of a regular $4n$-gon. 

\begin{thm}
\label{thm:MAIN}
For each $n \geq 6$, there exists a hyperbolic link $L_{4n}$ with $Sym(\Sth \setminus L_{4n}) \cong (D_{4n} \times \ZZ_2) \rtimes \ZZ_2$ and $Sym(\Sth, L_{4n}) \cong \ZZ_2 \times \ZZ_2$. In particular, $[Sym(\mathbb{S}^{3} \setminus L_{4n}) : Sym(\mathbb{S}^{3}, L_{4n})] = 8n \rightarrow \infty$ as $n \rightarrow \infty$.
\end{thm}

A slightly more precise version of Theorem \ref{thm:MAIN} is stated in Corollary \ref{thm:Symmetrygroupdisparity}. For this construction, we first examine a family of highly symmetric hyperbolic links: minimally twisted chain links with $4n$  components, the set of which we denote by $\{ C_{4n} \}_{n=2}^{\infty}$; see Figure \ref{fig:SymDiag}. It was shown in \cite{MeMiTr2020} that the symmetry groups of these links and their complements coincide and $|Sym(\Sth, C_{4n})| = 32n$. In Section \ref{sec:SymMinTwisted}, we review some background on minimally twisted chain links and the symmetry groups of these links and their respective complements, while providing a more explicit description of these symmetry groups. In Section \ref{sec:CP}, we describe a specific type of  homeomorphism of $\Sth \setminus C_{4n}$, called an $ml$-swap, which can be performed along any pair of punctured annuli in $\Sth \setminus C_{4n}$ bounded by a Hopf sublink of $C_{4n}$. The image of $C_{4n}$ under an $ml$-swap is a link in $\Sth$, where a link diagram can be explicitly described from a corresponding link diagram of $C_{4n}$.  By performing multiple well-chosen $ml$-swaps at once on $\Sth \setminus C_{4n}$, we  can construct a variety of links that are not isotopic to $C_{4n}$ but whose complements are homeomorphic to $\Sth \setminus C_{4n}$. In Section \ref{sec:SymCP} we show how to choose a set of  $ml$-swaps on  $\Sth \setminus C_{4n}$ that produces a non-isotopic link $L_{4n}$, where $|Sym(\Sth, L_{4n})|$ is universally bounded. Our control over this symmetry group comes from a basic combinatorial structure associated to the linking number of pairs of components of $L_{4n}$ that can be easily examined under $ml$-swaps and must be preserved under any symmetry of $(\Sth, L_{4n})$. We then expand upon these techniques to construct other families of links built by performing a sequence of $ml$-swaps on $\Sth \setminus C_{4n}$ and where we control the behavior of their corresponding symmetry groups. This is highlighted in Theorem \ref{thm:symcontrol} and Corollary \ref{cor:symcontrol}.

We would like to thank the referee from an earlier version of this paper for providing several helpful comments and suggestions that improved this paper.


\section{Symmetries of Minimally Twisted Chain Links}
\label{sec:SymMinTwisted}

A \textbf{minimally twisted $2n$-chain link}, denoted $C_{2n}$, is a link that consists of $2n$ unknotted circles embedded in $\Sth$ as a closed chain, where consecutive components alternate between lying in the projection plane and lying perpendicular to the projection plane. See Figure \ref{fig:SymDiag} for a diagram of $C_{8}$, which we refer to as its symmetric diagram. For $n \geq 3$, such links are known to be hyperbolic. This was first shown in Chapter 6 of Thurston's notes \cite[Example 6.8.7]{thurston:notes}, where Thurston described a particular decomposition of the hyperbolic structures of these link complements and determined an explicit formula for their volumes. Kaiser--Purcell--Rollins further studied the volumes of these links and other types of chain links in \cite{KPR2012}.

	Here, we are interested in the symmetries of minimally twisted $2n$-chain links and their complements. Such links are a special class of Montesinos links, whose symmetries were determined by Boileau and Zimmermann \cite{BoZi}. At the same time, these hyperbolic links have also been described in the literature as flat fully augmented pretzel links since $C_{2n}$ can be constructed by fully augmenting a pretzel link with $n$ twist regions and an even number of twists in each twist region; see \cite{Pu4} for an introduction to fully augmented links. Meyer--Millichap--Trapp \cite{MeMiTr2020} used the geometric structures on fully augmented links to study topological, geometric, and arithmetic properties of $\Sth \setminus C_{2n}$.  In particular, the symmetry groups $Sym(\Sth, C_{2n})$ and $Sym(\Sth \setminus C_{2n})$ where examined in  their work, which we now summarize and expand upon. 

\begin{figure}[ht]
	\centering
	\begin{overpic}[scale=1.0]{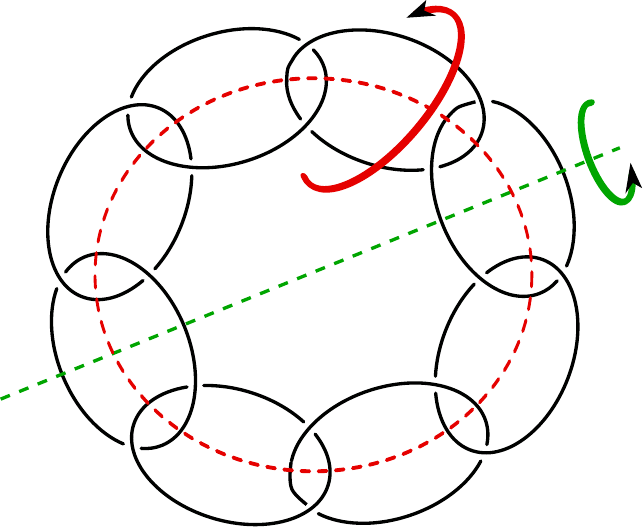}
		\put(73, 77){\LARGE{$\beta$}}
		\put(98,59){\LARGE{$\gamma$}}
	\end{overpic}
		\caption{A symmetric diagram for $C_{8}$.}
	\label{fig:SymDiag}
\end{figure}

The symmetries of $(\Sth, C_{2n})$ are all ``visually obvious'' as seen in the symmetric diagram given in Figure \ref{fig:SymDiag} for $C_{8}$.  For the remainder of this paper, assume the $2n$ components of $C_{2n}$ have been given a cyclic ordering by traversing this chain of link components in a clockwise direction in the symmetric diagram for $C_{2n}$ and assume these components have been labeled $K_{1}, \ldots, K_{2n}$.  Let $\beta$ be the $180^{\circ}$ rotation about the circular axis depicted in Figure \ref{fig:SymDiag}. Let $\alpha$ be the symmetry that maps $K_i$ to $K_{i+1}$ (mod $2n$) for $i=1, \ldots, 2n$, via a clockwise rotation about the center of the ring of components by $\frac{180}{n}^{\circ}$ followed by a  rotation of $90^{\circ}$ along the circular axis in the same direction as $\beta$, as depicted in Figure \ref{fig:SymDiag}. Let $\gamma$ be the $180^{\circ}$ rotation about the linear axis depicted in Figure \ref{fig:SymDiag}. As shown in  \cite[Theorem 6.1]{MeMiTr2020}, these three elements generate the group of orientation-preserving symmetries of this link. If we include $r$, the reflection across the projection plane, in our generating set, then we generate all of $Sym(\Sth, C_{2n})$, giving a group of order $16n$; see \cite[Corollary 6.3]{MeMiTr2020}.  Meyer--Millichap--Trapp used volume arguments to show that $Sym(\Sth, C_{2n}) = Sym(\Sth \setminus C_{2n})$. We highlight these results in the following theorem.

\begin{thm}[Theorem 6.1 and Corollary 6.3, \cite{MeMiTr2020}]
\label{thm:SymComp}
For $n \geq 4$, any minimally twisted $2n$-chain link has the following property:

\begin{center}
$Sym(\Sth \setminus C_{2n}) = Sym(\Sth, C_{2n})$,
\end{center}

where $Sym(\Sth,C_{2n})$ has order $16n$ and is generated by $\alpha$, $\beta$, $\gamma$, and $r$.
\end{thm}

Note that, while $C_{6}$ and $C_{8}$ are both hyperbolic links, they are also the only minimally twisted $2n$-chain links that are arithmetic. The techniques used in proving Theorem 6.1 in \cite{MeMiTr2020} were dependent on properties of non-arithmetic cusped hyperbolic $3$-manifolds. In particular, $C_{6}$ and $C_{8}$ (and their complements) might have additional symmetries beyond those given in Theorem \ref{thm:SymComp}. However, by checking  $C_{8}$ in SnapPy, we found that Theorem 6.1 and Corollary 6.3 of \cite{MeMiTr2020} still hold, and so, we included this case in  the statement of Theorem \ref{thm:SymComp}. 

At this point, we will restrict our analysis to minimally twisted $4n$-chain links. Our results could easily be generalized to minimally twisted $2n$-chain links, though the group decomposition of $Sym(\mathbb{S}^{3}, C_{2n})$ varies slightly, depending on whether $n$ is even or odd. When $n$ is even, we have that $\alpha^{2n}=1$, while when $n$ is odd, we have that $\alpha^{2n} = \beta$. Since the class $\{ C_{4n}\}$ suffices to prove our major results, we decided to exclude the latter case.

 Now, we use the action of $Sym(\Sth, C_{4n})$ on this link's symmetric diagram in Figure \ref{fig:SymDiag} to more explicitly describe the symmetry groups highlighted in Theorem \ref{thm:SymComp}.

\begin{prop}
\label{prop:symgroupdecomp}
For $n \geq 2$, we have that $Sym(\Sth, C_{4n}) = H_1 \rtimes H_2 \cong (D_{4n} \times \ZZ_2) \rtimes \ZZ_2$, where $H_1 = <\alpha, \beta, \gamma> \cong D_{4n} \times \ZZ_2$ and $H_2 = <r> \cong \ZZ_2$. 
\end{prop}

\begin{proof}
First, we can see that the action of the subgroup $<\alpha, \gamma>$ on $C_{4n}$ is equivalent to the action of $D_{4n}$ on a regular $4n$-gon $P_{4n}$, with components of $C_{4n}$ corresponding to vertices  of $P_{4n}$ and two linked components of $C_{4n}$ corresponding with an edge of $P_{4n}$. Thus $<\alpha, \gamma> \cong D_{4n}$. One can then check that $\beta \alpha = \alpha \beta$ and $\beta \gamma = \gamma \beta$, making $H_1 = <\alpha, \beta, \gamma> \cong D_{4n} \times \ZZ_2$. At the same time, $r$ does not commute with $\alpha$ since components embedded in the projection plane are fixed pointwise by $r$, while components perpendicular to the projection plane have their top halves reflected across the projection plane to their bottom halves via $r$. So, the action of $\alpha r$ on a component $K$ embedded in the projection plane is the same as the action of $\alpha$ on $K$, which is not the same as $r$ acting on the component $\alpha \cdot K$, which is perpendicular to the projection plane. This shows that $Sym(\Sth, C_{4n})$ is not a direct product of $H_1$ and $H_2$. Since $|H_1| = 16n$, we have that $[Sym(\Sth, C_{4n}): H_1] = 2$, and so, $H_1$ is a normal subgroup of $Sym(\Sth, C_{4n})$. Thus, since $Sym(\Sth, C_{4n}) = <\alpha, \beta, \gamma, r> = H_{1}H_{2}$, $H_{1} \cap H_{2} = \{ 1 \}$, and $H_{1} \trianglelefteq Sym(\Sth, C_{4n})$, we have that $Sym(\Sth, C_{4n}) = H_{1} \rtimes H_{2}$.
\end{proof}

The following corollary to  Proposition \ref{prop:symgroupdecomp} and Theorem \ref{thm:SymComp} will be useful in our analysis. First, we introduce some notation. Given a link $L \subset \Sth$, we let $Sym_{f}(\Sth, L)$ denote the subgroup of $Sym(\Sth, L)$ that maps each component of $L$ to itself setwise. Similarly, we define $Sym_{f}(\Sth \setminus L)$ to be the subgroup of $Sym(\Sth \setminus  L)$ that maps every cusp of $\Sth \setminus L$ to itself setwise. 

\begin{cor}
\label{cor:fixedsyms}
For $n \geq 2$, we have that $Sym_{f}(\Sth \setminus C_{4n}) = Sym_{f}(\Sth,C_{4n}) = <\beta, r> \cong \ZZ_{2} \times \ZZ_{2}$.
\end{cor}

\begin{proof}
Since every symmetry of $\Sth \setminus C_{4n}$ is induced by a symmetry of $(\Sth, C_{4n})$, we have that  $Sym_{f}(\Sth \setminus C_{4n}) = Sym_{f}(\Sth, C_{4n})$.  We can immediately see that  $<\beta, r>$  is a subgroup of $Sym_{f}(\Sth, C_{4n})$ which is  isomorphic to $\ZZ_{2} \times \ZZ_{2}$. Since the dihedral action of $< \alpha, \gamma>$ on $C_{4n}$ is faithful and $<\alpha, \gamma, \beta, r> = Sym(\Sth, C_{4n})$, we can conclude that there are no other elements in $Sym_{f}(\Sth, C_{4n})$, and so, $Sym_{f}(\Sth, C_{4n}) = <\beta, r>$.
\end{proof}


\section{Non-overlapping Nested Partners}
\label{sec:CP}

Let $L$ be a link in $\Sth$. We say that a link $L' \subset \Sth$ is a \textbf{complement partner} for $L$ if and only if $\Sth \setminus L \cong \Sth \setminus L'$. We let $CP(L)$ denote the set of all complement partners for $L$, up to isotopy. Theorem \ref{thm:SymComp} shows that $|
Sym(\Sth, C_{2n})| =  \max \{|Sym(\Sth, L')| \hspace{0.05in} : \hspace{0.05in} L' \in CP(C_{2n}) \}$. Here, we will show how to explicitly construct  an infinite class of links $\{L_{4n}\}$ with $L_{4n} \in CP(C_{4n})$, and then in Section \ref{sec:SymCP}, we will show that $Sym(\Sth, L_{4n})$ behaves quite differently than $Sym(\Sth, C_{4n})$.  Since the results in this section don't depend on the group decomposition given in Proposition \ref{prop:symgroupdecomp}, we work with minimally twisted $2n$-chain links here.

For our construction, we will use a particular type of homeomorphism of $\Sth \setminus C_{2n}$ that is a composition of three specific Dehn twists along annuli, which we call an $ml$-swap,  and  analyze it's induced action on $C_{2n}$.  The term $ml$-swap refers to the fact that the meridians and longitudes of a Hopf sublink are swapped under this homeomorphism. This was first introduced by  Zevenbergen in \cite{Ze2021} and further applications of $ml$-swaps can be found in \cite[Section 5]{MiTr2023}.

Recall that we have labeled the components of $C_{2n}$ as $K_1, \ldots, K_{2n}$ by following a cyclic ordering in our symmetric diagram for $C_{2n}$.  Each $K_{i}$ bounds a disk, $D_{i}$, in $\Sth$. Each $D_{i}$ is punctured once by $K_{i-1}$ and once by $K_{i+1}$ (mod $2n$).  Figure \ref{fig:DehnTwists} illustrates an $ml$-swap on the Hopf sublink $\mathcal{H}$ determined by a pair of consecutive link components in the chain $C_{2n}$, which we labeled $\{K_i,K_{i+1}\}$.   In this context we let $A_i$ be the annulus $D_i\setminus K_{i+1}$ and $A_{i+1} = D_{i+1}\setminus K_i$.

\begin{figure}[h]
\begin{center}
\includegraphics{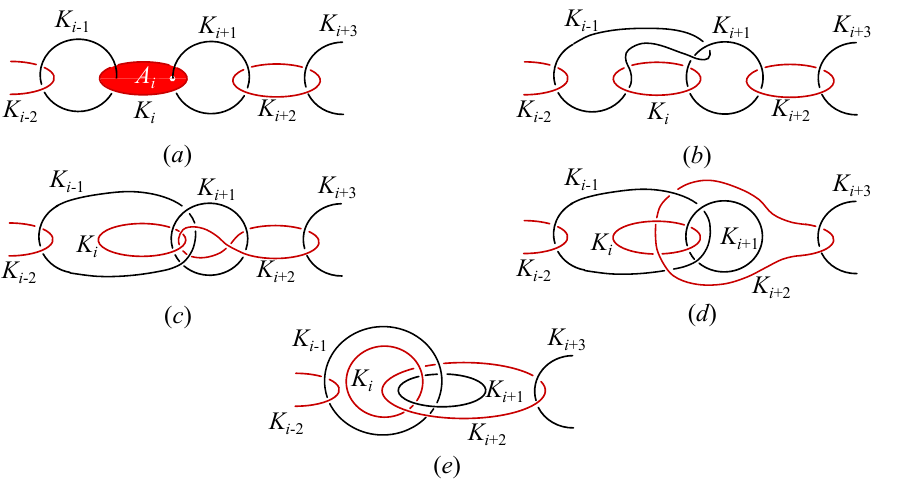}
\end{center}
\caption{An $ml$-swap as a product of annular Dehn twists}
\label{fig:DehnTwists}
\end{figure}

We now describe annular Dehn twists that realize Zevenbergen's $ml$-swap, and provide an explicit description of the link $C'_{2n}$ resulting from it.  For simplicity, assume the disks $D_i$ and $D_{i+1}$ are perpendicular to each other. Let $A_i$ and $A_{i+1}$ be the open annuli described above.   An $ml$-swap is the composition of annular Dehn twists along these annuli and, since $K_{i-1}$ and $K_{i+2}$ are the only components of $C_{2n}$ that intersect the twisting annuli, they are the only components to change their position relative to the others. First perform a right-handed Dehn twist along $A_i$, wrapping $K_{i-1}$ around $K_{i+1}$; see Figure \ref{fig:DehnTwists}$(b)$. Then perform a left-handed twist along $A_{i+1}$, which is punctured by both $K_{i-1}$ and $K_{i+2}$, and is a counterclockwise rotation about the $K_i$ puncture in Figure \ref{fig:DehnTwists}$(b)$. This left-handed Dehn twist unlinks $K_{i-1}$ from $K_i$ while linking $K_{i+2}$ with both $K_{i-1}$ and $K_i$; see Figure \ref{fig:DehnTwists}$(c)$. Figure \ref{fig:DehnTwists}$(d)$ illustrates that another right-handed Dehn twist along $A_i$ unlinks  $K_{i+1}, K_{i+2}$.  The resulting link can be isotoped until $K_i,K_{i-1},K_{i+3}$ are perpendicular to the projection plane while $K_{i-2},K_{i+1}, K_{i+2}$ lie on the projection plane (see Figure \ref{fig:DehnTwists}$(e)$). Since Dehn twists are homeomorphisms on the complement, the resulting link $C'_{2n}$ is a complement partner of $C_{2n}$.


\begin{defn}\label{defn:mlSwap}
The \textbf{$i^{th}$ $ml$-swap} of $\Sth \setminus C_{2n}$ is a homeomorphism between link complements  $s_i: \Sth \setminus C_{2n} \rightarrow \Sth \setminus C_{2n}'$ that is a composition of a right-hand Dehn twist along $A_{i}$, followed by a left-hand Dehn twist along $A_{i+1}$, followed by a right-hand Dehn twist along $A_{i}$, as described in Figure \ref{fig:DehnTwists}. 
\end{defn}

Since an $ml$-swap $s_i$ is a composition of three Dehn twists along the annuli $A_i, A_{i+1}$, this action only affects a sufficiently small neighborhood of $K_{i-1} \cup K_{i} \cup K_{i+1} \cup K_{i+2}$ (mod $2n$) in $\mathbb{S}^{3}$.  Thus, we say that two $ml$-swaps $s_i$ and $s_j$ of $\Sth \setminus C_{2n}$ are \textbf{non-overlapping} if and only if $|i-j| \geq 4$ (mod $2n$). Given a finite sequence of non-overlapping $ml$-swaps, their individual actions on $(\Sth, C_{2n})$ occur in pairwise disjoint neighborhoods, respectively. Thus, these $ml$-swaps commute with each other (as homeomorphisms of their complements) and we can easily describe the induced action of their composition on $(\Sth, C_{2n})$. This motivates the following definition.

\begin{defn}
We say that a link $L \subset \Sth$ is a \textbf{non-overlapping nested partner} of $C_{2n}$ if $L$ can be constructed by performing a finite sequence of non-overlapping $ml$-swaps on $C_{2n}$. We let $NNP(C_{2n})$  denote the set of such links, up to isotopy. 
\end{defn}

Describing an $ml$-swap in terms of annular Dehn twists is convenient for determining the specific link that results.  We now give a surgery description of $ml$-swaps which, in addition to illustrating more clearly that meridians and longitudes are swapped, will be useful when analyzing symmetry groups.

 To specify surgery instructions on a link $L=L_1\cup \cdots \cup L_n$ in $\Sth$, first choose disjoint closed tubular neighborhoods $N_i$ of $L_i$. Now specify a \textbf{slope} $C_i$ (an isotopy class of an unoriented simple closed curve) on each boundary torus $T_i=\partial N_i$. To perform Dehn surgery of $\Sth$ along $L$, with the given surgery instructions: first remove the interiors $\interior{N}_i$ of each $N_i$, resulting in the manifold $X=\Sth \setminus\left(\interior{N}_1\cup\cdots \cup \interior{N}_n\right)$ with boundary tori $T_i$.  Now attach solid tori $V_i$ to $X$ using boundary homeomorphisms $h_i:\partial V_i \to T_i$ that map meridians of $\partial V_i$ to the specified curves $C_i$. Let $M'$ denote the resulting manifold and note that the core longitudes of the $V_i$ form a link $L'$ in $M'$. There is ambiguity in choosing where the $h_i$ send longitudes of $\partial V_i$. Different choices for $h_i$ result in manifolds homeomorphic to $M'$, but the resulting links $L'$ can be distinct.  

With the above notation, note that any self-homeomorphism $h:X\to X$ determines an alternative surgery description of $M'$ simply by replacing each $C_i$ with $h(C_i)$.  This observation will be useful in providing a surgery description of Definition \ref{defn:mlSwap}, as well as in the proof of Theorem \ref{thm:symcontrol}.

If $L\subset \Sth$, preferred meridians $\mu_i$ and longitudes $\lambda_i$ of $L_i$ are used to parameterize slopes on $T_i$. More precisely, if $C_i = \pm(p\mu_i+q\lambda_i)$ we perform $p/q$-surgery on $L_i$.  For example, a $0$-surgery corresponds to choosing $C_i$ to be a longitude $\lambda_i$ while $\infty$-surgery is when $C_i=\mu_i$.  Thus surgery instructions on $L\subset \Sth$ is a labeling of each component of a link $L\subset \Sth$ with a rational number (or infinity).  

In order to specify the link $L'\subset M'$ resulting from a Dehn surgery on $L\subset \Sth$,  we make conventions on the gluing homeomorphisms $h_i$. If $L_i$ is labeled $0$, then the gluing homeomorphism $h_i:\partial V_i\to T_i$ maps a meridian of  $\partial V_i$ to a longitude of $T_i$.  For our purposes we choose $h_i$ to also map a longitude $\partial V_i$ to a meridian of $T_i$, so choose $h_i$ to swap meridians and longitudes.  Similarly, if $L_i$ is labeled $\infty$, choose $h_i$ to preserve meridians and longitudes.  

The next lemma gives a surgery description of $ml$-swaps and their resulting links, and we introduce some terminology.  Given a sublink $L^{\ast}\subset L$, labeling each component of a $L^{\ast}$ with $0$ and each component of $L\setminus L^{\ast}$ with $\infty$ will be called a $0$-surgery on $L^{\ast}$.  Note that surgery along each component with coefficient $\infty$ is trivial. 

\begin{lemma}\label{lem:MLSwap}
Let $L'\subset M'$ be the link resulting from $0$-surgery on $(K_i\cup K_{i+1})\subset C_{2n}$ in $\Sth$. Then $M' = \Sth$ and $L'=C_{2n}'$, where $C_{2n}'\subset \Sth$ is the link resulting from $ml$-swap $s_i$.  Moreoever, every $L'\in NNP(C_{2n})$ is the result of doing $0$-surgery on the sublink in $C_{2n}$ involved in the $ml$-swaps creating $L'$.
\end{lemma}

\begin{proof}
Let $M'$ be the manifold resulting from a $0$-surgery on the Hopf sublink $(K_i\cup K_{i+1})\subset C_{2n}$, and let $L'\subset M'$ be the link whose components are cores of the surgery tori.  Now let $N_1\cup \cdots\cup N_{2n}$ be a choice of disjoint closed tubular neighborhoods of the components of $C_{2n}$ in $\Sth$, chosen so that $N_i$ and $N_{i+1}$ are invariant under the Dehn twists along the annuli $A_i,A_{i+1}$ used in defining the $ml$-swap. Then the $ml$-swap $s_i$ restricts to a homeomorphism of $X=\Sth\setminus \interior{N}_i\cup \interior{N}_{i+1}$.  More precisely, $s_i|_X:X\to X$ is the homeomorphism resulting from the appropriate product of Dehn twists along the properly embedded annuli $A_i\cap X$ and $A_{i+1}\cup X$.

As noted above, the homeomorphism $s_i|_X$ determines an alternative surgery description for $M'$ with surgery slopes $s_i|_X(C_j)$ along boundary components $s_i|_X(T_j)$, for $j=i,i+1$.  Notice $s_i|_X$ fixes each boundary torus $T_i,T_{i+1}$ setwise, and thus it remains to compute images of the surgery slopes under $s_i|_X$.  

We focus  our attention to the action of $s_{i}|_{T_i}$ on slopes of $T_i$, as the argument for $T_{i+1}$ is similar. Note that Dehn twists along annuli $A_i,A_{i+1}$ restrict on $T_i$ to Dehn twists along the curves $A_i\cap T_i$ and $A_{i+1}\cap T_i$. By definition of the annuli $A_i,A_{i+1}$, they intersect the boundary torus $T_i$ in a preferred longitude and meridian of $K_i$, respectively.  Since the surgery cofficient of $K_i$ is zero, the surgery slope $C_i$ is a longitude of $K_i$, or $C_i = A_i\cap T_i$.  A quick calculation confirms that the defining Dehn twists for $s_i|_{T_i}$ swaps meridional and longitudinal slopes and, in particular, $s_i|_{T_i}(C_i)$ is a meridian of $K_i$.  Similarly, $s_i|_{T_i}(C_{i+1})$ is a meridian of $K_{i+1}$.

Since $s_i|_X$ maps both surgery curves to meridians, the new surgery coefficients for $K_i, K_{i+1}$ are both infinity. Hence $M'$ is homeomorphic to the manifold obtained by doing $\infty$ surgery on all components of $C_{2n}$ or, in other words, $M' = \Sth$.  The preceding discussion also shows that $L'=C_{2n}'$ as the homeomorphism used is the restriction of $s_i$ to $X$ and the new surgeries on $K_i,K_{i+1}$ are trivial.

It remains to extend this argument to non-overlapping nested partners.  Since the $ml$-swaps involved are non-overlapping, the annuli used to define different $ml$-swaps are disjoint.  One can then choose disjoint regular neighborhoods of each pair of intersecting annuli, and applying the above argument to each pair completes the proof.
\end{proof}

One nice feature about this surgery description is that it highlights the symmetry in an $ml$-swap that is somewhat hidden in the Dehn twist description.  Since the annuli aren't twisted the same number of times to perform an $ml$-swap, it may seem like $K_i$ plays a different role from $K_{i+1}$ in the process.  This, however, is misleading since the important property is to swap meridional and longitudinal slopes of both curves. The surgery description highlights this symmetry since both components of $\mathcal{H}_i = K_i\cup K_{i+1}$ have the same surgery coefficient.  A second appealing feature is that Lemma \ref{lem:MLSwap} can be used to relate symmetries of $C_{2n}$ and $L'$, as we now show.

In the following corollary we use the fact that $Sym(\Sth\setminus C_{2n}) \cong Sym(\Sth, C_{2n})$ to think of an element $\psi\in Sym(\Sth\setminus C_{2n})$ as permuting components of $C_{2n}$.

\begin{cor}\label{cor:MLSwapSymm}
Let $L'\in NNP(C_{2n})$, and let $\mathcal{H} =\{\mathcal{H}_{i_1},\dots,\mathcal{H}_{i_m}\}$ be the set of Hopf sublinks involved in the non-overlapping $ml$-swaps that create $L'$.  If $[\psi]\in Sym(\Sth\setminus C_{2n})$ preserves $\mathcal{H}$ set-wise, then $[\psi]$ extends to  an element of $Sym(\Sth,L')$.
\end{cor}

\begin{proof}
By Lemma \ref{lem:MLSwap}, components of $L'$ are core longitudes of the solid tori $V_i$ used to perform the $0$-surgery on the sublink $\mathcal{H}\subset C_{2n}$.  Using notation introduced prior to Lemma \ref{lem:MLSwap}, we have seen that $0$-surgery on $\mathcal{H}$ results in the filled manifold $M' = \Sth$, so $L'\subset \Sth$.  Recall that $X$ is the complement of open tubular neighborhoods of components of $C_{2n}$, and that each homeomorphism $h:X\to X$ induces new surgery instructions using surgery slopes $h(C_i)$ on $h(T_i)\subset \partial X$. Observe that if $h$ preserves surgery slopes, i.e., whenever $h(T_i) = T_j$ we have $h(C_i) = C_j$, then $h$ extends to a homeomorphism $h:M'\to M'$ of the manifold $M'$ resulting from the surgery.  The corollary follows from this observation.

Now suppose $[\psi]\in Sym(\Sth\setminus C_{2n})$ preserves $\mathcal{H}$ set-wise. The interior $\interior{X}$ is homeomorphic to $\Sth\setminus C_{2n}$, so a representative homeomorphism $\psi$ has a unique continuous extension to $X$, also denoted by $\psi:X\to X$.  The proof is complete once $\psi$ is shown to preserve surgery slopes.

Now observe that every generator, hence every element, of $Sym(\Sth,C_{2n})$ preserves slopes of preferred meridians and longitudes for each $K_i\in C_{2n}$ (recall that slopes are unoriented).  Since $Sym(\Sth\setminus C_{2n}) \cong Sym(\Sth, C_{2n})$, the same is true of $\psi$. Moreover, a $0$-surgery on $\mathcal{H}$ is determined by slopes that are either longitudinal  along $T_i$ (if $K_i\in \mathcal{H}$), or meridional otherwise.  Now $\psi: X\to X$ preserves $\mathcal{H}$, so it also preserves components of $C_{2n}\setminus\mathcal{H}$.  Thus $\psi: X\to X$ preserves surgery slopes, completing the proof. 
\end{proof}


\section{Symmetries of Non-overlapping Nested Partners of $C_{4n}$}
\label{sec:SymCP}

We would now like to analyze the structure of  $Sym(\Sth, L)$ for certain $L \in NNP(C_{4n})$.  The following definition will help with this analysis. 

\begin{defn}
Let $K$ and $K'$ be two distinct components of a link $L \subset \Sth$. We let $lk(K, K')$ denote the linking number of $K$ and $K'$, as defined in \cite{Ro1976}.  The \textbf{linking graph} for $L$, denoted $\Gamma(L)$, is the graph where each vertex corresponds with a component of $L$ and two vertices, corresponding to components $K, K' \subset L$, are connected by an edge exactly when $lk(K, K') \neq 0$. 
\end{defn}

Zevenbergen used this linking graph, and some additional linking criteria, to bound the number of distinct  links that are nested partners to $C_{2n}$ in \cite{Ze2021}. Here, we will relate symmetries of $(\Sth, L)$ to automorphisms of $\Gamma(L)$ via the induced  homomorphism $\phi: Sym(\Sth, L) \rightarrow Aut(\Gamma(L))$. Any homeomorphism $\rho$ of the pair $(\Sth, L)$ permutes the components of $L$ and preserves non-trivial linking number, i.e., if $K$ and $K'$ are components of $L$, then  $lk(K, K') \neq 0$ if and only if $lk(\rho(K), \rho(K')) \neq 0$. So, a homeomorphism of $(\Sth, L)$ induces an automorphism of $\Gamma(L)$. This map $\phi$ is well-defined 
 since linking number is invariant under isotopic homeomorphisms of $(\Sth, L)$. Thus, the action of $[\rho]$ on $L$ induces an automorphism $\phi([\rho])$ on $\Gamma(L)$.

We immediately see that the linking graph $\Gamma(C_{4n})$ is a single cycle with $4n$ vertices, each corresponding to a $K_i$ and which we label $v_i$. In this case, $Aut(\Gamma(C_{4n}))\cong D_{4n}$ and the homomorphism $\phi: Sym(\Sth, C_{4n}) \rightarrow Aut(\Gamma(C_{4n}))$ is onto since $\psi$ maps the subgroup $<\alpha,\gamma>$ isomorphically to $Aut(\Gamma(C_{4n}))$. To determine the effect of an $ml$-swap of $C_{4n}$ on its corresponding linking graph, compare (a) and (e) in Figure \ref{fig:DehnTwists}. 
From this diagrammatic description, we can construct $\Gamma(s_{i}(C_{4n}))$  from $\Gamma(C_{4n})$ by first switching labels between $v_{i+1}$ and $v_{i}$, and then adding a single edge connecting $v_{i-1}$ to $v_{i+2}$. When we perform a sequence of non-overlapping $ml$-swaps, this procedure generalizes, giving the following lemma. See Figure \ref{fig:NoSymEx} for an example of $L_{24} \in NNP(C_{24})$, where $L_{24}$ is constructed by performing $s_2 \circ s_8 \circ s_{16}$ on $C_{24}$. 

\begin{lemma}
\label{lem:SwapActionLG}
Let $n \geq 2$. For $L \in NPP(C_{4n})$, assume that $s_{i_{1}} \circ \ldots \circ s_{i_{m}}$ is the composition of non-overlapping $ml$-swaps performed on $(\Sth, C_{4n})$ to produce $(\Sth, L)$. Then the linking graph $\Gamma(L)$ is obtained from $\Gamma(C_{4n})$ by switching labels between $v_{i_{j}+1}$ and $v_{i_{j}}$ and inserting an edge connecting $v_{i_{j}-1}$ to $v_{i_{j}+2}$,  for each $j = 1, \ldots, m$, (mod $4n$).
\end{lemma}

Given $L \in CP(C_{4n})$, we will use $\phi: Sym(\Sth, L) \rightarrow Aut(\Gamma(L))$ to place restrictions on $Sym(\Sth, L)$.

\begin{lemma}
\label{lem:kernelbound}
Let $n \geq 2$. If $L \in CP(C_{4n})$, then $ker(\phi) = Sym_{f}(\Sth, L)$ is  isomorphic to a subgroup of $\ZZ_2 \times \ZZ_2$. If $L \in NNP(C_{4n})$, then $Sym_{f}(\Sth, L) \cong \ZZ_2 \times \ZZ_2$. 
\end{lemma}

\begin{proof}
 Let $L \in CP(C_{4n})$.  Then $Sym_{f}(\Sth, L) \subseteq ker(\phi)$ since any such $\rho \in Sym_{f}(\Sth, L)$ induces an automorphism of $\Gamma(L)$ that fixes the vertices pointwise, and since $\Gamma(L)$ is a simple graph by construction, this implies that the induced automorphism is trivial. The other inclusion immediately follows, and so, $ker(\phi) = Sym_{f}(\Sth, L)$.

 Since any symmetry of a link that fixes every component induces  a symmetry of the complement that fixes every cusp, we have that $Sym _{f}(\Sth, L)$ is a subgroup of $Sym_{f}(\Sth \setminus L)$. Then $ Sym_{f}(\Sth, L) \leq Sym_{f}(\Sth \setminus L) \cong Sym_{f}(\Sth \setminus C_{4n}) \cong \ZZ_2 \times \ZZ_2$, by Corollary \ref{cor:fixedsyms}. 
 
 If $L \in NNP(C_{4n})$, then reflection in the projection plane and the rotation of $180^{\circ}$ about the circular axis going through this chain (similar to $\beta$ in Figure \ref{fig:SymDiag}) each are distinct order two elements in $Sym_{f}(\Sth, L)$. This implies that $Sym_{f}(\Sth, L) \cong \ZZ_2 \times \ZZ_2$.
\end{proof}

We can now show that every minimally twisted $4n$-chain link admits a non-overlapping nested partner whose only symmetries are those that map every link component to itself.

\begin{figure}
\[
\begin{array}{ccc}
 \includegraphics[width=1.8in]{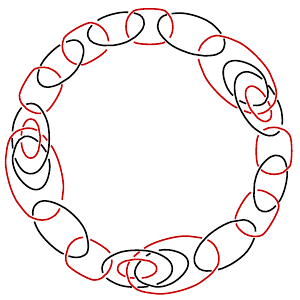} & \includegraphics[width=1.8in]{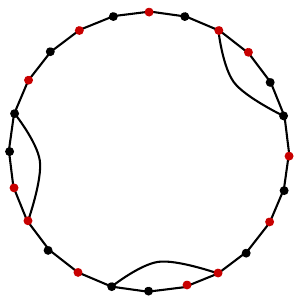}\\
 (a)\textrm{ The link }L_{24} & (b)\textrm{ Linking graph }\Gamma(L_{24})
\end{array}
\]
\caption{A non-overlapping nested partner of $C_{24}$ and its corresponding linking graph.}
\label{fig:NoSymEx}
\end{figure}

\begin{prop}
\label{prop:symlimitations}
For every $n \ge 6$, there exists $L_{4n} \in NNP(C_{4n})$ such that $$Sym(\Sth, L_{4n}) = Sym_{f}(\Sth, L_{4n}) \cong \ZZ_2 \times \ZZ_2.$$ 
\end{prop}

\begin{proof}
Let $n \geq 6$ and construct $L_{4n}$ by performing $s_{2} \circ s_{8} \circ s_{16}$ on $C_{4n}$. By definition, $L_{4n} \in NNP(C_{4n})$. By Lemma \ref{lem:SwapActionLG}, we see that $\Gamma(L_{4n})$ can be constructed by starting with a single $4n$-cycle with vertices labeled $v_1, \ldots, v_{4n}$ following a clockwise ordering and adding exactly three edges: $e_{1}$ connecting $v_1$ to $v_4$, $e_{2}$ connecting $v_7$ to $v_{10}$, and $e_{3}$ connecting $v_{15}$ to $v_{18}$. See Figure \ref{fig:NoSymEx} for a diagram of $L_{24}$ and $\Gamma(L_{24})$. For this proof, we ignore the relabeling introduced between certain vertices of $\Gamma(C_{4n})$ under $ml$-swaps since this does not affect our argument.

We claim that $Aut(\Gamma(L_{4n}))$ is trivial. To see this we consider degree-three vertices, and paths of degree-two vertices between them.  Any automorphism of $\Gamma(L_{4n})$ must permute the degree-three vertices. Thus, the edges $e_1$, $e_2$, and $e_3$ must also be permuted by an automorphism of $\Gamma(L_{4n})$ since they are the only edges connecting a pair of degree-three vertices. This implies that any such automorphism maps an adjacent pair (a pair of vertices connected by $e_i$, for $i=1,2,3$) to an adjacent pair.  Since the paths between adjacent pairs of degree-three vertices are different lengths, any automorphism of $\Gamma(L_{4n})$ preserves adjacent pairs of degree-three vertices.  Further, adjacent degree-three vertices have paths of different lengths emanating from them, and so, each such pair cannot be swapped in a graph automorphism.  Thus degree-three vertices are point-wise fixed, which implies the remaining vertices are fixed as well.  Hence the automorphism group of $\Gamma(L_{4n})$ is trivial. 

Since $\phi(Sym(\Sth, L_{4n})) \subset Aut(\Gamma(L_{4n})) \cong \{ 1\}$, we know that $Sym(\Sth, L_{4n}) / \ker(\phi)$ is the trivial group. Thus, $Sym(\Sth, L_{4n}) = \ker(\phi) = Sym_{f}(\Sth, L_{4n})$. By Lemma \ref{lem:kernelbound}, we then have that $Sym(\Sth, L_{4n}) \cong \ZZ_2 \times \ZZ_2$. 
\end{proof}

The following corollary shows that our non-overlapping nested partners $\{ L_{4n} \}_{n=6}^{\infty}$ all have  isomorphic symmetry groups, while the symmetry groups of their complements grow linearly with $n$. This gives an explicit construction of links whose symmetry groups are arbitrarily small as subgroups of the symmetry groups of their corresponding complements. A slightly less general version of this corollary is stated in the introduction as Theorem \ref{thm:MAIN}.

\begin{cor}
\label{thm:Symmetrygroupdisparity}
For every $n \geq 6$, let $L_{4n} \in NNP(C_{4n})$ be the link constructed in Proposition \ref{prop:symlimitations}. Then $Sym(\Sth \setminus L_{4n}) \cong (D_{4n} \times \ZZ_2) \rtimes \ZZ_2$ and $Sym(\Sth, L_{4n}) \cong \ZZ_2 \times \ZZ_2$. In particular, $[Sym(\mathbb{S}^{3} \setminus L_{4n}) : Sym(\mathbb{S}^{3}, L_{4n})] = 8n \rightarrow \infty$ as $n \rightarrow \infty$.
\end{cor}

\begin{proof}
Since $\Sth \setminus L_{4n} \cong \Sth \setminus C_{4n}$, we have that $Sym(\Sth \setminus L_{4n}) \cong Sym(\Sth \setminus C_{4n}) \cong (D_{2n} \times \ZZ_2) \rtimes \ZZ_2$ by Proposition \ref{prop:symgroupdecomp}.  At the same time, Proposition \ref{prop:symlimitations} shows that $Sym(\Sth, L_{4n}) \cong \ZZ_2 \times \ZZ_2$. The statement on subgroup index follows from $|\ZZ_2 \times \ZZ_2| = 4$ and $|(D_{4n} \times \ZZ_2) \rtimes \ZZ_2| = 32n$.
\end{proof}

We can more broadly apply the techniques used in proving Proposition \ref{prop:symlimitations} and Corollary \ref{thm:Symmetrygroupdisparity} to determine symmetry groups of many other non-overlapping nested partners of $C_{4n}$.

In what follows, recall that $\alpha, \beta, \gamma, r$ are generators for $Sym(\Sth \setminus C_n) \cong Sym(\Sth \setminus L)$, for any $L \in CP(C_{n})$. In particular, we can consider $Sym(\Sth, L)$ as a subgroup of $<\alpha, \beta, \gamma, r >$. 

\begin{thm}
\label{thm:symcontrol}
Let $n \geq 2$. For $L \in NNP(C_{4n})$, we have that  $$Sym(\Sth, L) = Sym_{f}(\Sth, L) \rtimes H \cong (\ZZ_2 \times \ZZ_2) \rtimes  (Aut(\Gamma(L)),$$
where $H = <\alpha, \gamma> \cap Sym(\Sth, L)$. 
\end{thm}

\begin{proof}
We first show that $Sym(\Sth, L)$ has the desired group structure. Since $Sym_{f}(\Sth, L) = ker(\phi)$, we know $Sym_{f}(\Sth, L)$ is a normal subgroup of $Sym(\Sth, L)$. Furthermore, Lemma \ref{lem:kernelbound} shows that $Sym_{f}(\Sth, L) = <\beta, r> \cong \mathbb{Z}_{2} \times \mathbb{Z}_{2}$. Then $H$ can act upon the normal subgroup $Sym_{f}(\Sth, L)$ via conjugation, which provides an inner semidirect product structure $Sym_{f}(\Sth, L) \rtimes H \leq Sym(\Sth, L)$. Since $Sym(\Sth \setminus L) = < \beta, r> \rtimes <\alpha, \gamma>$ and $Sym(\Sth, L) \leq Sym(\Sth \setminus L)$, we also have the reverse inclusion. Thus, $Sym(\Sth, L) = Sym_{f}(\Sth, L) \rtimes H$. 

We now justify that $H$ is isomorphic to $Aut(\Gamma(L))$. Since $H = <\alpha, \gamma> \cap Sym(\Sth, L)$ and $<\alpha, \gamma>$ acts faithfully, we see that any two distinct elements of $H$, $a_1$ and $a_2$, provide different permutations of the components of $L$. Since the vertices of $\Gamma(L)$ correspond with the components of $L$, this implies that $\phi(a_1) \neq \phi(a_2)$, and so, $\phi$ restricted to $H$ is one-to-one. 

To finish the proof, we show that $\phi$ maps $H$ onto $Aut(\Gamma(L))$. The strategy is to show that given any $f \in Aut(\Gamma(L))$, we can construct $F \in Sym(\Sth\setminus C_{2n})$ that induces $f$ and preserves the Hopf sublinks used to construct $L$. An application of Corollary \ref{cor:MLSwapSymm} then completes the argument.

Since $L \in NNP(C_{4n})$, there is a set of disjoint Hopf sublinks $\mathcal{H} = \{H_1\cup H_2\cup \cdots \cup H_m\}\subset C_{4n}$ such that $L$ is the result of a $0$-surgery on $\mathcal{H}$. Recall from Lemma \ref{lem:SwapActionLG} that $\Gamma(L)$ is obtained from the $4n$-cycle $\Gamma(C_{4n})$ by switching labels between $v_{i+1}$ and $v_{i}$ and then inserting an edge connecting $v_{i-1}$ to $v_{i+2}$,  for each $ml$-swap $s_i$ performed. Thus, the vertices of $\Gamma(C_{4n})$ corresponding to components of each $H_j \in \mathcal{H}$ are those vertices whose labels are switched in constructing the unique $4n$-cycle of $\Gamma(L)$.  Let $f \in Aut(\Gamma(L))$, and let $f^{\ast}$ be the automorphism in  $Aut(\Gamma(C_{4n}))$ that agrees with $f$ on vertices that aren't swapped.  Recall that $Aut(\Gamma(C_{4n})) \cong D_{4n}$, and the dihedral action of an element of this group on $\Gamma(C_{4n})$ is determined by its action on two consecutive vertices. Now $f$ and $f^{\ast}$ must agree on at least two consecutive vertices since $L \in NNP(C_{4n})$.  Thus, $f^{\ast}$ is unique.

More can be said about $f^{\ast}$.  Since $f \in Aut(\Gamma(L))$, it must permute the edge set $\{e_j\}_{j=1}^{m}$  added to $\Gamma(C_{4n})$ to construct $\Gamma(L)$.  Since only one side of each $e_j$ has exactly two vertices, they must correspond to components of $H_j\subset \mathcal{H}$.  Thus both $f$ and $f^{\ast}$ permute the vertices corresponding to components in the $H_i$.  

Now there is an $F\in Sym(\Sth\setminus C_{4n})$ with $\phi(F) = f^{\ast}$, since $\phi: Sym(\Sth\setminus C_{4n}) \to Aut(\Gamma(C_{4n}))$ is onto. Moreover, since $f^{\ast}$ preserves vertices corresponding to components in $\mathcal{H}$, the symmetry $F$ preserves $\mathcal{H}$ set-wise.  Corollary \ref{cor:MLSwapSymm} implies that $F$ extends to a symmetry $F'\in Sym(\Sth,L)$.  By construction, $F'$ induces $f\in Aut(\Gamma(L))$, completing the proof that $H\cong Aut(\Gamma(L))$. 
\end{proof}

\begin{remark}
Note that, our hypothesis of $L \in NNP(C_{4n})$ is necessary to have the desired group decomposition given in Theorem \ref{thm:symcontrol}. For instance, one could construct $L' \in CP(C_{4n}) \setminus NNP(C_{4n})$ by performing $\frac{1}{k_{i}}$ Dehn surgeries along all components of $C_{4n}$ with even indices to create $2n$ $(2, 2k_i)$-torus sublinks. One can chose $k_{i}$ so that the linking numbers are all different, and so, only the symmetry denoted by $\beta$ will remain, as torus links are chiral. This observation was communicated to us by a referee of an earlier draft of this paper. 
\end{remark}

We provide a corollary showing how we can use particular choices of $ml$-swaps to control the symmetry group of a non-overlapping nested partner of $C_{4n}$.  In what follows, $D_1 \cong \ZZ_2$ and $D_2 \cong \ZZ_2 \times \ZZ_2$.

\begin{cor}
\label{cor:symcontrol}
Let $n \geq 2$. If $k$ divides $4n$ and $ 1 \leq k \leq n$, then there exists $L \in NNP(C_{4n})$ such that $$Sym(\Sth, L)  \cong (\ZZ_2 \times \ZZ_2) \rtimes D_k.$$ 
\end{cor}

\begin{proof}
Let $n \geq 2$ and assume $k$ divides $4n$ with $1 \leq k \leq n$. By Theorem \ref{thm:symcontrol}, we just need to show that there exists $L \in NNP(C_{4n})$ such that $Aut(\Gamma(L)) \cong D_{k}$. Given $C_{4n}$, perform the composition of $ml$-swaps $s_1 \circ s_{\frac{4n}{k} +1} \circ s_{2(\frac{4n}{k}) +1} \circ \cdots \circ s_{(k-1)\frac{4n}{k} + 1}$. Since $k \leq n$, we have that $\frac{4n}{k} \geq 4$. Thus, these $ml$-swaps are pairwise non-overlapping, and so,  $L \in NNP(C_{4n})$. By Lemma \ref{lem:SwapActionLG}, $\Gamma(L)$ is obtained by adding $k$ edges $\{e_i \}_{i=1}^{k}$ to $\Gamma(C_{4n})$ and relabeling certain vertices. By construction, each such $e_i$ connects a pair of vertices of $\Gamma(C_{4n})$ that were a distance $3$ apart in $C_{4n}$, and there is a distance of $\frac{4n}{k} -3$ from one pair of valence three vertices corresponding to the same $ml$-swap to the next closest such pair of vertices in $\Gamma(L)$, following the cyclic ordering of the vertices $C_{4n}$. Since any automorphism of $\Gamma(L)$ must map the set $\{e_i\}_{i=1}^{k}$ to itself, and these $k$ edges are evenly spaced in $\Gamma(L)$, we have that $Aut(\Gamma(L)) \cong D_k$. 
\end{proof}

We also note an interesting immediate application of Corollary \ref{cor:symcontrol} suggested by the referee:

\begin{cor}
	Let $k_1, \ldots, k_p \in \mathbb{N}_{\geq 2}$ and set $n = k_1 \cdots k_p$. Then there exists a family of links $\{L_i\}_{i=1}^{p}$ such that $Sym(\Sth, L_i)  \cong (\ZZ_2 \times \ZZ_2) \rtimes D_{k_{i}}$ and $\Sth \setminus L_i \cong \Sth \setminus C_{4n}$, for all $i=1,\ldots, p$. 
\end{cor}

One could easily construct other conditions on $ml$-swaps, and more broadly, Dehn twists along components of $C_{4n}$ to  obtain further results about symmetry groups of complement partners of $C_{4n}$. These constructions lead one to ask: which subgroups of $Sym(\Sth \setminus C_n)$ can be realized as the symmetry group of some complement partner of $C_n$? This question was addressed by Belleman \cite{Belleman2023} in an unpublished manuscript. A more general version of this question is found in Question \ref{Q1}.

\section{Concluding Remarks and Open Questions}
\label{sec:concluding}

Theorem \ref{thm:symcontrol} and Corollary \ref{cor:symcontrol} highlight the variety of symmetry groups that can occur just for nonoverlapping nested partners of minimally twisted $4n$-chain links and how one can control such symmetry groups using particular types of homeomorphisms. There are several ways these results could be expanded upon or generalized. One could examine symmetry groups of nested (not necessarily non-overlapping) partners or complement partners of $C_{2n}$. In another direction, one could consider symmetry groups of a larger class of links and their respective complement partners. As mentioned in Section \ref{sec:SymMinTwisted}, minimally twisted $4n$-chain links are a special subclass of (flat) fully augmented links, a set of links whose complements have exceptionally nice geometric structures that can be leveraged in the study of symmetries. In particular, their link complements contain many totally geodesic surfaces that can be used to restrict the behavior of any symmetry of their respective complements; see \cite{MiTr2023}. For instance, in \cite{Ag3}, Agol leveraged these tools to show that Dehn fillings of flat fully augmented chain mail links provide a class of asymmetric L-space links with arbitrary many components. Thus, the techniques used in this paper combined with the tractable geometry of fully augmented link complements have the potential to help answer several questions about symmetries of links and their complements. 

We conclude this paper with some questions in this direction.

\begin{question}
	\label{Q1}
Given a hyperbolic link $L$ and any complement partner $L'$, we know that $Sym(\Sth, L') \subseteq Sym(\Sth \setminus L)$. For a fixed $L$, can one  classify which subgroups of $Sym(\Sth \setminus L)$ are realized as symmetry groups of complement partners?
\end{question}

\begin{question}
	\label{Q2}
Recall that the symmetry group of a hyperbolic link is a finite subgroup of $O(4)$ and such subgroups have been classified. Thus, can one determine which of these subgroups are realized as $Sym(\Sth, F)$, for some fully augmented link $F$? 
\end{question}


\bibliographystyle{hamsplain}
\bibliography{biblio}

\providecommand{\bysame}{\leavevmode\hbox to3em{\hrulefill}\thinspace}
\providecommand{\href}[2]{#2}
\begin{thebibliography}{10}

\bibitem{Ag3}
Ian Agol, \emph{Chainmail links and l-spaces},  (2023),
  \mbox{arXiv:2306.10918}.

\bibitem{Belleman2023}
Annette Belleman, \emph{Symmetries of augmented links: On fals, chain links,
  and their fal-equivalent counterparts}, CSUSB REU Program (2023),
  \mbox{https://reuatcsusb.github.io/}.

\bibitem{BoZi}
Michel Boileau and Bruno Zimmermann, \emph{Symmetries of nonelliptic
  {M}ontesinos links}, Math. Ann. \textbf{277} (1987), no.~3, 563--584.

\bibitem{GL}
C.~McA. Gordon and J.~Luecke, \emph{Knots are determined by their complements},
  J. Amer. Math. Soc. \textbf{2} (1989), no.~2, 371--415.

\bibitem{Gu:thesis}
Fran{\c{c}}ois Gu{\'e}ritaud, \emph{G{\'e}om{\'e}trie hyperbolique effective et
  triangulations id{\'e}ales canoniques en dimension trois}, PhD Thesis,
  University of Paris-XI, Orsay, 2006 (2006).

\bibitem{HeWe1992}
Shawn~R. Henry and Jeffrey~R. Weeks, \emph{Symmetry groups of hyperbolic knots
  and links}, J. Knot Theory Ramifications \textbf{1} (1992), no.~2, 185--201.

\bibitem{KPR2012}
James Kaiser, Jessica~S. Purcell, and Clint Rollins, \emph{Volumes of chain
  links}, J. Knot Theory Ramifications \textbf{21} (2012), no.~11, 1250115, 17.

\bibitem{Lu1992}
Feng Luo, \emph{Actions of finite groups on knot complements}, Pacific J. Math.
  \textbf{154} (1992), no.~2, 317--329.

\bibitem{MeMiTr2020}
Jeffrey~S. Meyer, Christian Millichap, and Rolland Trapp, \emph{Arithmeticity
  and hidden symmetries of fully augmented pretzel link complements}, New York
  J. Math. \textbf{26} (2020), 149--183.

\bibitem{MiTr2023}
Christian Millichap and Rolland Trapp, \emph{Flat fully augmented links are
  determined by their complements},  (2023), \mbox{arXiv:2302.02002}.

\bibitem{MiWo2016}
Christian Millichap and William Worden, \emph{Hidden symmetries and
  commensurability of 2-bridge link complements}, Pacific J. Math. \textbf{285}
  (2016), no.~2, 453--484.

\bibitem{Mo1984}
John~W. Morgan, \emph{History of the {S}mith conjecture and early progress},
  The {S}mith conjecture ({N}ew {Y}ork, 1979), Pure Appl. Math., vol. 112,
  Academic Press, Orlando, FL, 1984, pp.~7--9.

\bibitem{PaPo2009}
Luisa Paoluzzi and Joan Porti, \emph{Hyperbolic isometries versus symmetries of
  links}, Topology Appl. \textbf{156} (2009), no.~6, 1140--1147.

\bibitem{Pu4}
Jessica~S. Purcell, \emph{An introduction to fully augmented links},
  Interactions between hyperbolic geometry, quantum topology and number theory,
  Contemp. Math., vol. 541, Amer. Math. Soc., Providence, RI, 2011,
  pp.~205--220.

\bibitem{Ro1976}
Dale Rolfsen, \emph{Knots and links}, Mathematics Lecture Series, No. 7,
  Publish or Perish, Inc., Berkeley, Calif., 1976.

\bibitem{SaWe}
Makoto Sakuma and Jeffrey Weeks, \emph{Examples of canonical decompositions of
  hyperbolic link complements}, Japan. J. Math. (N.S.) \textbf{21} (1995),
  no.~2, 393--439.

\bibitem{thurston:notes}
William~P. Thurston, \emph{Geometry and topology of 3-manifolds, lecture
  notes}, Princeton University (1978).

\bibitem{Ze2021}
Matthew Zevenbergen, \emph{Crushtaceans and complements of fully augmented and
  nested links}, Honors Thesis at the University of Rochester (2021).

\end{thebibliography}

\end{document}